\documentclass[12pt]{amsart}
\usepackage{amsfonts}
\usepackage{graphicx}
\usepackage{tabularx}
\usepackage{array}
\usepackage[usenames,dvipsnames]{color}
\usepackage{comment}
\usepackage{amsmath}
\usepackage{amsthm}
\usepackage{amssymb}
\usepackage{fullpage}
\usepackage[dvipsnames]{xcolor}
\usepackage{tikz}
\usepackage{listings}

\DeclareMathOperator{\dist}{dist}

\DeclareMathOperator{\chio}{\chi_o}
\DeclareMathOperator{\chis}{\chi_2}
\DeclareMathOperator{\exponential}{exp}

\newtheorem{theorem}{Theorem}[section]

\newtheorem*{conjecture}{Conjecture}

\newtheorem{lemma}[theorem]{Lemma}

\theoremstyle{definition}

\def\epsilon{\varepsilon}

\title{On Oriented Colourings of Graphs on Surfaces}

\author{Alexander Clow}
\address{Department of Mathematics, Simon Fraser University}
\email{alexander\_clow@sfu.ca}

\date{\today}

\begin{document}

\begin{abstract}
For an oriented graph $G$, the least number of colours required to oriented colour $G$ is called the oriented chromatic number of $G$ and denoted $\chio(G)$.
For a non-negative integer $g$ let $\chio(g)$ be the least integer such that $\chio(G) \leq \chio(g)$ for every oriented graph $G$ with Euler genus at most $g$. 
We will prove that $\chio(g)$ is nearly linear in the sense that $\Omega(\frac{g}{\log(g)}) \leq \chio(g) \leq O(g \log(g))$. This resolves a question of the author, Bradshaw, and Xu, by improving their bounds of the form $\Omega((\frac{g^2}{\log(g)})^{1/3}) \leq \chio(g)$ and $\chio(g) \leq O(g^{6400})$.
\end{abstract}

\maketitle

\section{Introduction}

In this paper we consider \emph{oriented graphs}, which are directed graphs without loops or multi-edges.
If $H$ is an oriented graph and $G$ is the underlying simple graph of $H$, then we say \emph{$H$ is an orientation of $G$}.
If $H$ is an orientation of $G$ and we consider a parameter of $H$ that does not depend on orientation, then suppose we are considering this parameter in $G$. 
For example, let the maximum degree of $H$ be the maximum degree of $G$.

Given oriented graphs $G=(V,E)$ and $H = (V',E')$ we say a map $\phi: V \rightarrow V'$ is an \emph{oriented homomorphism} from $G$ to $H$ if for all $(u,v)\in E$, there exists an edge $(\phi(u),\phi(v))\in E'$. 
See Figure~\ref{Fig: o-colouring example} for an example of an oriented homomorphism.
Given an oriented homomorphism $\phi$ from $G$ to $H$, we say that the partition of $V$ given by $\{\phi^{-1}(z): z \in V'\}$ is an \emph{oriented colouring} of $G$. If $H$ has at most $k$ vertices, then we say $\{\phi^{-1}(z): z \in V'\}$ is an oriented $k$-colouring of $G$, and if $G$ admits an oriented homomorphism to an oriented graph with at most $k$ vertices, then we say $G$ is oriented $k$-colourable. The \emph{oriented chromatic number} of an oriented graph $G$, denoted $\chio(G)$ is the least integer $k$ such that $G$ is oriented $k$-colourable. If $G$ is a simple graph, then the oriented chromatic number of $G$, denoted $\chio(G)$, is the maximum oriented chromatic number of any orientation of $G$.

\begin{figure}[h]
\centering
\scalebox{0.75}{
\begin{tikzpicture}[node distance={15mm}, ultra thick, main/.style = {draw, circle}, minimum size = 0.1cm] 

\node[main][fill= blue] (1) at (0,0) {}; 
    \node[fill=none] at (0,-0.5) (nodes) {$3$};
\node[main][fill= red] (2) at (3,0) {}; 
    \node[fill=none] at (3,-0.5) (nodes) {$1$};
\node[main][fill= cyan] (3) at (6,0) {};
    \node[fill=none] at (6,-0.5) (nodes) {$2$};
\node[main][fill= blue] (4) at (4.5,2.5) {}; 
    \node[fill=none] at (5,2.5) (nodes) {$3$};
\node[main][fill= red] (5) at (3,5) {};
    \node[fill=none] at (3,5.5) (nodes) {$1$};
\node[main][fill= cyan] (6) at (1.5,2.5) {}; 
    \node[fill=none] at (1,2.5) (nodes) {$2$};
\node[main][fill= green] (7) at (3,1.5) {};
    \node[fill=none] at (3.25,2.25) (nodes) {$4$};

\draw [->] (1) -- (2);
\draw [->] (2) -- (3);
\draw [->] (3) -- (4);
\draw [->] (4) -- (5);
\draw [->] (5) -- (6);
\draw [->] (6) -- (1);
\draw [->] (1) -- (7);
\draw [->] (2) -- (7);
\draw [->] (3) -- (7);
\draw [->] (4) -- (7);
\draw [->] (5) -- (7);
\draw [->] (6) -- (7);

\draw[->] (7.5,2) -- (10,2);

\node[main][fill= cyan] (i) at (11.5,0) {}; 
    \node[fill=none] at (11.5,-0.5) (nodes) {$2$};
\node[main][fill= blue] (ii) at (14.5,0) {}; 
    \node[fill=none] at (14.5,-0.5) (nodes) {$3$};
\node[main][fill= red] (iii) at (13,3) {};
    \node[fill=none] at (13,3.5) (nodes) {$1$};
\node[main][fill= green] (iv) at (13,1) {}; 
    \node[fill=none] at (13.5,1.5) (nodes) {$4$};

\draw [->] (i) -- (ii);
\draw [->] (ii) -- (iii);
\draw [->] (iii) -- (i);
\draw [->] (i) -- (iv);
\draw [->] (ii) -- (iv);
\draw [->] (iii) -- (iv);
\end{tikzpicture}}
    \caption{An oriented homomorphism and its associated oriented colouring.}
    \label{Fig: o-colouring example}
\end{figure}
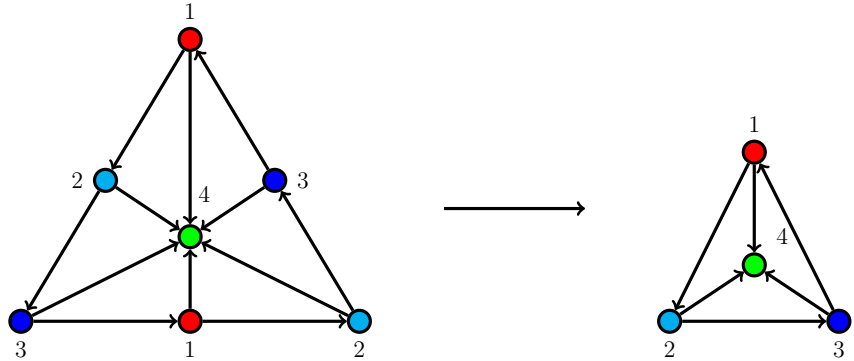

It is often convenient to consider $\chio(G)$ maximised over all graphs $G$ in a given family $\mathcal{G}$. In particular, this paper will study the family of graphs with Euler genus at most $g$. To this end we let $\chio(g)$ be the maximum oriented chromatic number of any graph with Euler genus at most $g$.

Oriented colouring was first introduced by Courcelle \cite{courcelle1994monadic} in the context of second order logic. In the same year Raspaud and Sopena \cite{raspaud1994good} introduced the problem in its current form when they proved that every planar graph has oriented chromatic number at most $80$. 
Equivalently, Raspaud and Sopena showed $\chio(0)\leq 80$.
The main step in Raspaud and Sopena's proof is to show that for all graphs $G$, $\chio(G) \leq \chi_a(G)2^{\chi_a(G)-1}$ where $\chi_a(G)$ is the acyclic chromatic number. 
Given for all planar graphs $\chi_a(G) \leq 5$, see Borodin \cite{borodin1979acyclic}, $\chio(0)\leq 80$ follows immediately.
Meanwhile, the best lower bound for $\chio(0)$ is of the form $\chio(0)\geq 18$  \cite{marshall2015oriented}.
Despite significant effort, see \cite{sopena2016homomorphisms}, the upper bound $\chio(0)\leq 80$ has yet to be improved.

From here determining upper bounds on $\chio(g)$ for $g>0$ was of natural interest.
Using the same approach as \cite{raspaud1994good}, along with a bound on the acyclic chromatic number of graphs on surfaces from Albertson and Berman \cite{albertson1978acyclic},
Kostochka, Sopena, and Zhu \cite{kostochka1997acyclic} observed that $\chio(g) \leq (4g+4)2^{4g+3}$. 
By taking a better bound on the acyclic chromatic number of graphs on surfaces by Alon, Mohar, and Sanders \cite{alon1996acyclic} this can be improved to $\chio(g) \leq 2^{O(g^{4/7})}$.

Using a more general forbidden subgraph colouring method Aravind and Subramanian \cite{aravind2013forbidden} we able to show that for every $\epsilon>0$ there exists a constant $c$ such that $\chio(g) \leq 2^{cg^{(1/2)+\epsilon}}$. 
Aravind and Subramanian \cite{aravind2013forbidden} then conjectured that there exists a constant $C$ such that
$$
\chio(g) \leq 2^{C\sqrt{g}}.
$$
This conjecture is the natural limit of the forbidden subgraph colouring arguments used in \cite{aravind2013forbidden,kostochka1997acyclic}.
This is because all forbidden subgraph colouring are proper colourings, implying that there exists graphs with Euler genus $g$ and forbidden subgraph chromatic number $\Omega(\sqrt{g})$ for any lists of forbidden subgraphs.

By applying an approach using that injective chromatic index, a kind of non-proper edge colouring, the author, Bradshaw, and Xu \cite{bradshaw2023injective} proved Aravind and Subramanian's conjecture. In fact, the author, Bradshaw, and Xu proved the much stronger results that
$$
\Omega((\frac{g^2}{\log(g)})^{1/3}) \leq \chio(g) \leq O(g^{6400}).
$$
As this provides a polynomial upper and lower bound for $\chio(g)$ in terms of $g$, this leads to the natural question posed in \cite{bradshaw2023injective}; what is the least $k$ such that $\chio(g) = O(g^k)$?

Our main result is the following theorem. 
As we are mostly concerned with the asymptotic order of $\chio(g)$ we do not make an effort to optimise the coefficients in our bounds.
If the base of a logarithm is not given, assume it is the natural logarithm. 

\begin{theorem}\label{Thm: Main Thm}
    For all integers $g \geq 11$,
    $$
    \frac{\log(2)(g-1)}{\log(g-1)+\log(\log(2)) - \log(2)} \leq \chio(g) \leq  2^{40} g \log(g).
    $$
\end{theorem}

Notice that this essentially answers the question from \cite{bradshaw2023injective} by showing that for all $\epsilon>0$, $\chio(g) = O(g^{1+\epsilon})$, while ensuring that there is no $r < 1$ such that $\chio(g) = O(g^r)$.
That is, the answer to the question from \cite{bradshaw2023injective} is either $1$ or there is no smallest constant with this property.

In Section~2 we define the notation required for the rest of the paper.
The proof of Theorem~\ref{Thm: Main Thm} will be divided into two parts. 
The lower bound is proven in Section~3. 
The proof of the upper bound is given in Section~4 and Section~5.
We conclude with a discussion of future work.

\section{Preliminaries}

In this section we provide the definitions and notations used in the rest of the paper.
We assume the reader is familiar with standard notation in graph theory. 
For readers not familiar with standard graph theoretic notation see \cite{West1996}.

A \emph{surface} is a 
connected compact Hausdorff space which is locally homeomorphic to an open disc in the plane. 
It is well known, see \cite{mohar2001graphs}, that for every surface $S$, there exists a graph $G$ with a \emph{$2$-cell embedding in $S$}, that is, an embedding $\Pi:G \rightarrow S$ such that each connected component of $S \setminus \Pi(G)$ is homeomorphic to an open disc in the plane.
If $S$ is a surface and $G$ is a graph with $n$ vertices and $e$ edges which has an $2$-cell embedding in $G$ with $f$ faces, then the \emph{Euler genus} of $S$ is the quantity $2 - n + e - f$.
The \emph{Euler genus} of a graph $G$ is the minimum value $g$ such that $G$ has an embedding on a surface of Euler genus $g$. 
In particular, a planar graph has Euler genus $g = 0$, and a graph embeddable on the projective plane has Euler genus $g \leq 1$, and a graph embeddable on the torus or Klein bottle has Euler genus $g \leq 2$.

We borrow the following notation and definition from \cite{clow2023oriented}.
Given an oriented graph $G$, a vertex $v \in V(G)$, and an ordered vertex set 
$U = \{u_1, \dots, u_d\} \subseteq N(v)$, 
we write $F(U,v,G)$ for the vector in $\{-1,1\}^d$
whose $i^{\text th}$ entry is $1$ if $(v,u_i)$ is an edge of $G$, and whose $i^{\text th}$ entry is $-1$ if $(u_i,v)$ is an edge of $G$. 

Now, suppose $H$ is an oriented $k$-partite graph with exactly $N$ vertices in each partite set.
Let the partite sets of $H$ be called $P_1, \dots, P_k$.
We say that $H$ is \emph{$(k,d,N)$-full} if 
the following holds:
for each value $i \in  [k]$,
each
ordered subset $U = \{u_1, \dots, u_{t}\} \subseteq \bigcup_{j \neq i} P_j$ of size $t \leq d$, and each vector $\vec{v} \in \{-1,1\}^t$,
there exists
a vertex $x \in P_i$
such that $F(U,x,H) = \vec{v}$.
For an example of a $(k,d,N)$-full graph see Figure~\ref{Fig:(2,2)-full drawing}.

\begin{figure}[h!]
\begin{center}
    \scalebox{0.75}{
        \begin{tikzpicture}[node distance={15mm}, ultra thick, main/.style = {draw, circle}, minimum size = 0.1cm] 

\node[main][fill= black] (+1) at (0,2) {}; 
\node[main][fill= black] (+2) at (0,4) {}; 
\node[main][fill= black] (+3) at (10,2) {}; 
\node[main][fill= black] (+4) at (10,4) {};

\node[main][fill= black] (-1) at (5,0) {}; 
\node[main][fill= black] (-2) at (5,2) {}; 
\node[main][fill= black] (-3) at (5,4) {}; 
\node[main][fill= black] (-4) at (5,6) {}; 

\draw  [->] (+1) -- (-1);
\draw  [->] (+1) -- (-2);
\draw  [->] (-3) -- (+1);
\draw  [->] (-4) -- (+1);

\draw  [->] (-1) -- (+2);
\draw  [->] (+2) -- (-2);
\draw  [->] (+2) -- (-3);
\draw  [->] (-4) -- (+2);

\draw  [->] (-1) -- (+3);
\draw  [->] (-2) -- (+3);
\draw  [->] (+3) -- (-3);
\draw  [->] (+3) -- (-4);

\draw  [->] (+4) -- (-1);
\draw  [->] (-2) -- (+4);
\draw  [->] (-3) -- (+4);
\draw  [->] (+4) -- (-4);
        \end{tikzpicture}
    }
\end{center}
\caption{An example of a $(2,2,4)$-full graph.}
\label{Fig:(2,2)-full drawing}
\end{figure}
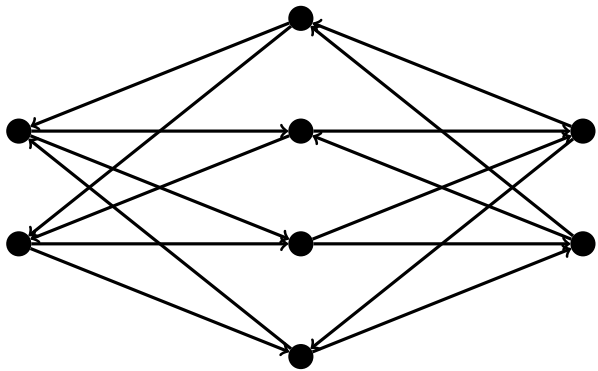

Given an oriented graph $G = (V,E)$ a \emph{$2$-dipath colouring} $\phi: V \rightarrow \mathbb{N}$ is a proper colouring such that for each pair of vertices $u,v$ if $1\leq \dist(u,v) \leq 2$, then $\phi(u) \neq \phi(v)$.
Here distance is directed distance. We say $\phi$ is a $2$-dipath $k$-colouring if $\phi$ is a $2$-dipath colouring and the range of $\phi$ has cardinality at most $k$. Then the $2$-dipath chromatic number of $  G$, denoted $\chis(G)$, is the least integer $k$ such that $G$ admits a $2$-dipath $k$-colouring. 
The $2$-dipath chromatic number of a simple graph is the maximum $2$-dipath chromatic number of any of its orientations.
First proposed by Chen and Wang \cite{min2006dipath}, a $2$-dipath colouring is equivalent to a proper colouring of the directed square, $G^2$, of the graph $G$.

Notice that every oriented colouring is a $2$-dipath colouring. This is because if $\phi$ is a colouring such that there exists a $2$-dipath $u,v,w$ with  $\phi(u) = \phi(w)$, then $\phi$ is not an oriented colouring.
This implies that for all oriented graphs $\chis(G) \leq \chio(G)$. 

A graph (or oriented graph) $G = (V,E)$ is an \emph{oriented clique} if and only if $\chio(G) = |V|$. It is not hard to verify that an oriented graph $G$ is an oriented clique if and only if the directed diameter of $G$ is at most $2$. Observe that this implies an oriented graph $G = (V,E)$ is an oriented clique if and only if $\chis(G)=|V|$.

\section{Oriented Cliques with Small Euler Genus}

This section is devoted to proving that $\chio(g) \geq \Omega(\frac{g}{\log g})$. Our proof relies on an existing result of Kostochka, Luczak, Simonyi, and Sopena from \cite{kostochka1997minimum} which bounds the minimum number of edges in an oriented clique of order $n$. From here we use Euler's formula to bound the Euler genus of such an oriented clique $G$.

The following Lemma is a weaker version of Theorem~2 from \cite{kostochka1997minimum}. As we are focused on the asymptotic order of $\chio(g)$ this weaker result is sufficient.

\begin{lemma}[Theorem~2 \cite{kostochka1997minimum}]\label{Lemma: Kostochka et. al extremal}
    Let $f(n)$ be the least integer such that there exists a graph $G$ on $n$ vertices and $f(n)$ edges with $\chio(G) = n$. If $n \geq 5$, then 
    $$f(n) \leq n\log_2(n).$$
\end{lemma}

We now proceed to bound the Euler genus of a graph of order $n$ with approximately $kn$ edges. Given Lemma~\ref{Lemma: Kostochka et. al extremal} we will take $k = \log_2(n)$ when proving Theorem~\ref{Thm: Near Linear Lower Bound}.

\begin{lemma}\label{Lemma: Euler's Formula Linear Edges}
    Let $G = (V,E)$ be a graph of order $n$ with at most $kn$ edges. Then, letting $g$ be the Euler genus of $G$
    $$
    g \leq (k-1)n+1.
    $$
\end{lemma}

\begin{proof}
Let $G = (V,E)$ be a graph of order $n$ with at most $kn$ edges and let $g$ be the Euler genus of $G$. Then Euler's formula implies that
\begin{eqnarray*}
    2 - g & \geq & |V| - |E| + 1\\
    & \geq & (1-k)n + 1.
\end{eqnarray*}
Hence, $g \leq (k-1)n+1$ as desired. This concludes the proof.
\end{proof}

As the functions we will deal with in the proof of Theorem~\ref{Thm: Near Linear Lower Bound} are somewhat awkward to solve without the proper background, we take a moment to mention the Lambert $W$ function, sometimes called the product log function. For reals $x\geq 0$ and $y$, the Lambert $W_0$ function, is the solution in $y$ to the equation $x = ye^y$. That is, if $x = ye^y$ and $x \geq 0$ is real, then $y = W_0(x)$. The Taylor series for $W_0(x)$ is given by 
\begin{eqnarray*}
    W_0(x) & := & \sum_{n=1}^\infty \frac{(-n)^{n-1}}{n!}.
\end{eqnarray*}
Additionally, as this will be useful when dealing with $W_0$, we point out that for reals $x \geq e$, $W_0(x) \geq \log(x) - \log\log(x)$ \cite{hoorfar2008inequalities}. Here $\log(x)$ denotes the natural logarithm of $x$.

With these tools in hand we are prepared to prove the lower bound $\chio(g) = \Omega(\frac{g}{\log(g)})$. In order to do this we prove the following stronger statement.

\begin{theorem}\label{Thm: Near Linear Lower Bound}
 For all integers $g \geq 11$, there exists a graph $G$ with Euler genus at most $g$ such that
 $$
 \chio(G) \geq \frac{\log(2)(g-1)}{\log(g-1)+\log(\log(2)) - \log(2)}.
 $$
\end{theorem}

\begin{proof}
Let $g \geq 11$ be a fixed integer
and let $n$ be the greatest integer satisfying that 
$$
g \geq  (\log_2(n)-1)n + 1.
$$
As $g \geq 11$, we conclude that $n \geq 5$. Let $G = (V,E)$ be a graph with as few edges as possible subject to the constrain that $|V| = n = \chio(G)$. Hence, by Lemma~\ref{Lemma: Kostochka et. al extremal}
$$
    |E| \leq n\log_2(n).
$$
Then, letting $g(G)$ be the Euler genus of $G$, Lemma~\ref{Lemma: Euler's Formula Linear Edges} implies that
$$
    g(G) \leq (\log_2(n)-1)n+1.
$$
So, by our choice of $n$, $G$ has Euler genus $g(G) \leq g$.
Also observe that as we chose $n$ to be as large as possible, $g < (\log_2(n+1)-1)(n+1) + 1$. 
Solving this inequality for $n$, with computer assistance, implies that $n$ is bounded below as a function of $g$. 
In particular,
\begin{eqnarray*}
    n & > & 2 \exponential \Big{(} W_0 \Big{(}  \frac{(g-1)\log(2)}{2}   \Big{)}   \Big{)} - 1 \\
    & \geq & 2 \exponential \Bigg{(}  \log \Big{(}  \frac{(g-1)\log(2)}{2} \Big{)}  - \log\log  \Big{(}  \frac{(g-1)\log(2)}{2}   \Big{)}    \Bigg{)} - 1 \\
    & = & \frac{\log(2)(g-1)}{\log(g-1)+\log(\log(2)) - \log(2)} - 1    
\end{eqnarray*}
where the second inequality follows from the fact that for $g \geq 11$, $\frac{(g-1)\log(2)}{2} > e$, and for all reals $x \geq e$, $W_0(x) \geq \log(x) - \log\log(x)$.

Recall that $\chio(G) = n$. 
Then, $G$ is a graph of Euler genus at most $g$ such that
$$
\chio(G) > \frac{\log(2)(g-1)}{\log(g-1)+\log(\log(2)) - \log(2)} - 1
$$
as desired.
\end{proof}

\section{Constructing Oriented Homomorphisms}

In this section we will prove that every oriented graph with Euler genus at most $g$ has an oriented homomorphism to some graph in a particular family of oriented graphs.
We will later see in Section~5 that each graph in this family has order $O(g\log(g))$.
The first order of business then is to describe this family of target graphs.

Let $g\geq 2$ be a fixed but arbitrary integer.
Let $K^{(g)}$ be a $(144g -162,10,N)$-full graph, where $N$ is chosen to be as small as possible.
We have not yet proven $K^{(g)}$ exists, this will be done in Section~5, suppose that $K^{(g)}$ exists for now.
As $K^{(g)}$ is $(144g -162,10,N)$-full, the underlying graph is a complete $144g -162$ partite graph, where each partite set has order $N$.
Let $P_1,\dots, P_{144g -162}$ be a fixed labeling of the partite sets of $K^{(g)}$.

We define $H_g$ as the subgraph of $K^{(g)}$ given by deleting all edges between vertices $u \in P_i$ and $v\in P_j$ for $i,j> 138g -162$.
More formally, 
$$H_g = K^{(g)} - E[K^{(g)}[P_{138g -161} \cup  \dots \cup P_{144g -162}]].$$
For all $1 \leq i \leq 138g -162$ we let $P_i$ denote the same set of vertices in $H_g$ as in $K^{(g)}$, while we let the union of all sets $P_i$ where $i > 138g -162$ be denoted $P_0$. 
The partition $P_0, P_1\dots, P_{138g -162}$ of $V(H_g)$ will be used throughout this section.

We are now prepared to state the main result of this section. 
We state the theorem we are trying to prove at the beginning of the section, then prove it at the end of the section because we will use the discharging method.

\begin{theorem}\label{Thm: Colouring Upper Bound}
    Let $g\geq 2$ be an integer.
    If $G = (V,E)$ is an oriented graph with Euler genus at most $g$, then there exists an oriented graph $H$ such that $H_g$ is a spanning subgraph of $H$ and $G$ has an oriented homomorphism to $H$.
\end{theorem}

Suppose throughout this section that $g\geq 2$ is a fixed integer.
We will prove
Theorem~\ref{Thm: Colouring Upper Bound} for this choice of $g$ by the discharging method. 
In service of this we will prove various properties of a smallest counterexample to Theorem~\ref{Thm: Colouring Upper Bound}.
To increase readability,
if we say that an oriented graph $G$ is a \emph{smallest counterexample}, then suppose that $G$ has Euler genus at most $g$, 
$G$ is a counterexample to Theorem~\ref{Thm: Colouring Upper Bound}, and
for all graphs $G'$ satisfying that
\begin{itemize}
    \item $G'$ has order strictly less than $G$, and $G'$ has Euler genus at most $g$, or
    \item $G'$ is a proper subgraph of $G$,
\end{itemize}
there exists an oriented graph $H$ such that $H_g$ is a spanning subgraph of $H$ and $G'$ has an oriented homomorphism to $H$. 
That is, $G'$ is not a counterexample to Theorem~\ref{Thm: Colouring Upper Bound} for this choice of $g$.

We begin by proving a useful property about oriented graphs $H$ which contain $H_g$ as a spanning subgraph.

\begin{lemma}\label{Lemma: full H}
    Let $H$ be an oriented graph that contains $H_g$ is a spanning subgraph.
    For all $1 \leq i \leq 138g-162$, and $A \subseteq V(H)$ such that $|A| \leq 10$, and $\vec{v} \in \{-1,1\}^{|A|}$, 
    if $A\cap P_i = \emptyset$, then there exists an $s\in P_i$ such that $\vec{v} = F(A,s,H)$.
\end{lemma}

\begin{proof}
     Let $H$ be an oriented graph that contains $H_g$ is a spanning subgraph, $A \subseteq V(H)$ such that $|A| \leq 10$, and let $\vec{v} \in \{-1,1\}^{|A|}$.
     We claim that for $1 \leq i \leq 138g-162$ if $A \cap P_i = \emptyset$, then there exists an $s\in P_i$ such that $\vec{v} = F(A,s,H)$.

     Recall that $H_g$ is a spanning subgraph of $H$ and 
     $$
     H_g = K^{(g)} - E[K^{(g)}[P_0]]
     $$
     for a $(144g -162,10,N)$-full graph $K^{(g)}$, whose partite sets are $P_1,\dots,P_{144g -162}$.
     Then, for each $z \in P_i$, the out-neighbours of $z$ in $H$ are a subset of the out-neighbours of $z$ in $K^{(g)}$, and similarly 
     the in-neighbours of $z$ in $H$ are a subset of the in-neighbours of $z$ in $K^{(g)}$.
     Hence, if $\vec{v} = F(A,s,K^{(g)})$ for a vertex $s \in P_i$, then $\vec{v} = F(A,s,H)$.

     As we have assumed $A \cap P_i = \emptyset$ and $|A| \leq 10$, our choice of $K^{(g)}$ as a $(144g -162,10,N)$-full graph implies that there exists an $s\in P_i$ such that $\vec{v} = F(A,s,K^{(g)})$.
     We have already seen that this implies $\vec{v} = F(A,s,H)$ given $1\leq i \leq 138g-162$.
     This completes the proof.
\end{proof}

\begin{lemma}\label{Lemma: <=3-vertex}
    If $G=(V,E)$ is a smallest counterexample,
    then $\delta(G)\geq 4$.
\end{lemma}

\begin{proof}
    Let $G=(V,E)$ is a smallest counterexample such that $\delta(G) \leq 3$.
    Then, there exists a vertex $v\in V$ such that $\deg(v)\leq 3$.
    If  $G[N(v)]$ is a clique, then $G' = G - v$.
    Otherwise let $G'$ be an oriented give by adding edges between the neighbours of $v$ in $G-v$ so that $G'[N(v)]$ is a clique.

    As $\deg(v) \leq 3$, $G'$ has Euler genus at most $g(G)\leq g$ where $g(G)$ is the Euler genus of $G$.
    Given $G$ is a smallest counterexample and $G'$ has order strictly less than $G$, there exists an oriented graph $H$ such that $H_g$ is a spanning subgraph of $H$ and $G'$ has an oriented homomorphism to $H$.
    Let $\phi'$ be an oriented homomorphism from $G'$ to $H$.

    As $G'[N(v)]$ is a clique, for all vertices $u,w \in N(v)$, $\phi'(u) \neq \phi'(w)$.
    Let $A = \{a_1,\dots, a_{\deg(v)}\}$ be an ordering of the set $\{\phi'(u): u\in N(v)\}$.
    Then, $|A|\leq 3$ as $\deg(v)\leq 3$.
    
    Let $\vec{v} \in \{-1,1\}^{|A|}$ with $x^{th}$ entry $1$ if $(v,u)\in E$, where $\phi'(u) = a_x$, and $x^{th}$ entry $-1$ otherwise.
    As $138g-162 -3 > 0$ and $|A|\leq 3$, there exists a $P_i$ for some $1\leq i \leq 138g-162$ such that $A \cap P_i = \emptyset$.
    Thus, Lemma~\ref{Lemma: full H} implies there exists a vertex $s\in V(H)$ such that $\vec{v} = F(A,s,H)$.
    Let $s$ be such a vertex

    Let $\phi:V \rightarrow V(H)$ be defined so that $\phi(v) = s$ and for all $u \in V\setminus \{v\}$, $\phi(u) = \phi'(u)$.
    Then, $\phi$ is an oriented homomorphism from $G$ to $H$, where $H$ contains $H_g$ as a spanning subgraph.
    But this contradicts $G$ being a smallest counterexample. 
    We conclude that if $G$ is a smallest counterexample, then $\delta(G) \geq 4$.
\end{proof}

\begin{lemma}\label{Lemma: ajacent low degree vertices}
If $G=(V,E)$ is a smallest counterexample,
then for all vertices $v\in V$, such that $4 \leq \deg(v) \leq 5$, every vertex $w\in N(v)$ satisfies $\deg(w) \geq 12$.
\end{lemma}

\begin{proof}
Suppose $G$ is a smallest counterexample such that there exists adjacent vertices $v,w\in V$ where $4 \leq \deg(v) \leq 5$ and $\deg(w) < 12$.
Let $e$ be the edge indecent to $v$ and $w$.
As $G$ is a smallest counterexample $G' = G-e$ is a not a counterexample.
Then there exists an oriented graph $H$ such that $H_g$ is a spanning subgraph of $H$ and $G'$ has an oriented homomorphism to $H$.
Let $H$ be such an oriented graph and $\phi'$ an oriented homomorphism from $G'$ to $H$.

We will define an oriented homomorphism $\phi$ from $G$ to $H$. 
For all $u \in V \setminus \{v,w\}$ let $\phi(u) = \phi'(u)$.
As $\phi'$ is an oriented homomorphism from $G'$ to $H$ for all $x,y \in V \setminus \{v,w\}$ if $x,v,y$ or $x,w,y$ is a $2$-dipath in $G$, then $\phi'(x) \neq \phi'(y)$ implying that $\phi(x) \neq \phi(y)$.
Let $A = \{a_1,\dots, a_{|A|}\}$ be an ordering of $\{\phi(u): u \in N(v)\setminus \{w\}\}$ and let $B = \{b_1,\dots, b_{|B}\}$ be an ordering of $\{\phi(u): u \in N(w)\setminus \{v\}\}$.

As vertices with different orientations to $w$ receive different colours, i.e for all $x,y \in V \setminus \{v,w\}$, $\phi(x)\neq \phi(y)$ if $x,w,y$ is a $2$-dipath in $G$, the vector $\vec{w} \in \{-1,1\}^{|B|}$ with $x^{th}$ entry $1$ if $(w,u)\in E$, where $\phi(u) = b_x$, and $x^{th}$ entry $-1$ otherwise, is well defined. 
Observe that for any $s \in V(H)$ such that $\vec{w} = F(B,s,H)$, we may let $\phi(w) = s$ and the resulting map $\phi$ is an oriented homomorphism from $G-v$ to $H$.
Here we consider $G-v$ as $\phi(v)$ has not yet been defined.

Recall that $|B| \leq \deg(w)-1 \leq 10$.
As $H_g$ is a spanning subgraph of $H$ Lemma~\ref{Lemma: full H} implies that
for all $1 \leq i \leq 138g-162$ such that $B \cap P_i = \emptyset$, there exists an $s \in P_i$ where $\vec{w} = F(B,s,H)$.
As $|B| \leq 10$, there are at least $138g-162-10 > 4$ vertices $s \in V(H)$ where $\vec{w} = F(B,s,H)$.
Hence, as $|A| \leq 4$ there exists a vertex $z \in V(H)$ such that $\vec{w} = F(B,z,H)$ and $z \notin A$.
Let $\phi(w)=z$ for such a vertex $z$.

Now let $\vec{v} \in \{-1,1\}^{|A|+1}$ be the vector with, for all $1\leq x \leq |A|$, $x^{th}$ entry $1$ if $(v,u)\in E$, where $\phi(u) = a_x$, and $x^{th}$ entry $-1$ otherwise, while the $(|A|+1)^{th}$ entry of $\vec{v}$ is equal to $1$ if $(v,w)\in E$, and equal to $-1$ if $(w,v)\in E$. 
As vertices with different orientations to $v$ receive different colours, i.e for all $x,y \in V$, $\phi(x)\neq \phi(y)$ if $x,v,y$ is a $2$-dipath in $G$, $\vec{v}$ is well defined.
Hence, for any $s \in V(H)$ such that 
$$
\vec{v} = F(\{a_1,\dots, a_{|A|},z=\phi(w)\},s,H),
$$
we may let $\phi(v) = s$ and the resulting map $\phi$ is an oriented homomorphism from $G$ to $H$.

Recall that $|A|+1 \leq \deg(v) \leq 5$.
As $H$ contains $H_g$ as a spanning subgraph Lemma~\ref{Lemma: full H} implies that
for all $1 \leq i \leq 138g-162$ such that $\{a_1,\dots, a_{|A|},z\} \cap P_i = \emptyset$, there exists an $s \in P_i$ where $\vec{v} = F(\{a_1,\dots, a_{|A|},z\},s,H)$. 
Thus, $138g-162 - 5 > 0$ implies that there exists a vertex $s\in V(H)$ such that $\vec{v} = F(\{a_1,\dots, a_{|A|},z\},s,H)$.
Let $\phi(v)=s$ for such a vertex $s$.
We have shown $G$ has an an oriented homomorphism to $H$

Therefore, $G$ has an oriented homomorphism to an oriented graph $H$ which contains $H_g$ as a spanning subgraph.
But this contradicts $G$ being a smallest counterexample.
We conclude that if $G$ is a smallest counterexample, then for all vertices $v\in V$, such that $4 \leq \deg(v) \leq 5$, every vertex $w\in N(v)$ satisfies $\deg(w) \geq 12$.
\end{proof}

The rest of the lemmas in this section will construct oriented homomorphisms by a similar, yet distinct, greedy strategy.
This new approach will use $2$-dipath colouring in a similar manner to \cite{bradshaw2023injective,clow2023oriented}.

\begin{lemma}\label{Lemma: Small Graph are easy to "(k,d,N)-colour"}
    If $G = (V,E)$ is an oriented graph of order $n \leq 6g$, then $G$ has an injective oriented homomorphism $\phi$ to an oriented graph $H$ which contains $H_g$ as a spanning subgraph, such that the image of $\phi$ is a subset of $P_0$.
\end{lemma}

\begin{proof}
    Let $G = (V,E)$ be an oriented graph of order $n \leq 6g$. 
    Recall that $H_g[P_0]$ is an oriented graph with no edges and $|P_0| \geq 6gN \geq 6g$.
    Let $\{u_1,\dots, u_{6g}\} \subseteq P_0$ and let $v_1,\dots, v_n$ be a fixed ordering of the vertices in $G$.

    Let $\phi: V \rightarrow V(H_g)$ such that for all $1\leq i \leq n \leq 6g$, $\phi(v_i)=u_i$. Then $\phi$ is an injective function and the image of $\phi$ is a subset of $P_0$. Form the graph $H$ by adding edges $(u_i,u_j)$ to $H_g$ for all edges $(v_i,v_j) \in E$. As $G$ is an oriented graph and $H_g[P_0]$ is an oriented graph with no edges the resulting graph $H$ is an oriented graph.

    Thus, $H$ is an oriented graph which contains $H_g$ as a spanning subgraph and $\phi$ is an injective oriented homomorphism from $G$ to $H$, such that 
    the image of $\phi$ is a subset of $P_0$.
    This completes the proof.
\end{proof}

The following lemmas bound the $2$-dipath chromatic number of oriented graphs with bounded maximum degree and degeneracy.
This will be critical to proving Lemma~\ref{Lemma: smallest counterexample large max degree}.
We note that Lemma~\ref{Lemma: smallest counterexample large max degree} is the key observation that will allow us to find a contradiction during the proof of Theorem~\ref{Thm: Colouring Upper Bound}.

\begin{lemma}\label{Lemma: 2-Diapth Max Degree & Degen Upper Bound}
    If $G = (V,E)$ is an oriented graph with maximum degree $\Delta$ and degenerate $d$, then 
    $$
    \chis(G) \leq 2d\Delta - \Delta - d^2 + d + 1.
    $$
\end{lemma}

\begin{proof}
    Let $G = (V,E)$ be a graph with maximum degree $\Delta$ and degenerate $d$. By the definition of the degenerate, there exists a vertex ordering $v_1,\dots,v_n$ of $V$, such that $|N(v_i) \cap \{v_1,\dots, v_{i-1}\}|\leq d$ for all $1 \leq i \leq n$. 
    Fix $i$ and suppose without loss of generality that $k = |N(v_i) \cap \{v_1,\dots, v_{i-1}\}|$.
    This implies that there are at most $k(\Delta-1)$ vertices $v_j$, 
    such that $v_j$ is a neighbour of a vertex $v_t \in N(v_i) \cap \{v_1,\dots, v_{i-1}\}$.
    Furthermore, there are at most $(d-1)(\Delta-k)$ vertices $v_j$, $j < i$, such that $v_j$ is a neighbour of a vertex $v_q \in N(v_i) \cap \{v_{i+1},\dots, v_{n}\}$.
    Thus, $v_i$ has at most $k(\Delta-1) + (d-1)(\Delta-k) = (d+k)\Delta-\Delta-dk$ vertices at undirected distance $2$ in $\{v_1,\dots, v_{i-1}\}$.

    Now we will colour $G$ greedily as follows. Suppose 
    $$\phi:\{v_1,\dots, v_{i-1}\} \rightarrow \{1,\dots, 2d\Delta - \Delta - d^2 + d + 1\}$$ 
    is a $2$-dipath colouring of $G[\{v_1,\dots, v_{i-1}\}]$ such that for all vertices $v_j$ where $j\geq i$, if $v_t$ and $v_r$ are both neighbours of $v_j$ and $r,t<i$, then $\phi(v_r)\neq \phi(v_t)$. 
    We will prove that we can extend $\phi$ to colour $v_i$ so that it remains a $2$-dipath colouring and satisfies this additional constraint.

    By assumption, for all $v_t,v_r \in N(v_i)\cap \{v_1,\dots, v_{i-1}\}$, $\phi(v_t)\neq \phi(v_r)$. Hence, adding $v_i$ to $G[\{v_1,\dots, v_{i-1}\}]$ will not create any $2$-dipaths between vertices of the same colour under $\phi$. Thus, if $v_i$ receives a colour that is distinct from any of the already coloured vertices, that is vertices in $\{v_1,\dots, v_{i-1}\}$, at undirected distance $1$ or $2$ from $v_i$ then this extended version of $\phi$ satisfies our assumptions.

    Recall that $v_i$ has $k \leq d$ already coloured neighbours and at most $(d+k)\Delta-\Delta-dk$ already coloured vertices at undirected distance $2$. Hence, if there are at least $\ell$ colours available to colour $v_i$, where
    $$
    \ell > (d+k)\Delta-\Delta-dk + k
    $$
    then there exists a colour we can assign $v_i$ to extend $\phi$ to colour $G[\{v_1,\dots, v_{i}\}]$. By taking a first derivative in $k$ and recalling that $\Delta\geq d$, we can see that $(d+k)\Delta-\Delta-dk + 2k$ is maximised when $k = d$. Hence, if we have at least $\ell$ colours where 
    $$
    \ell > 2d\Delta-\Delta-d^2 + d
    $$
    then there exists a colour we can assign $v_i$ to extend $\phi$. As we have assumed we have at least $2d\Delta-\Delta-d^2 + d+1$ colours we can extend our colouring $\phi$ to $v_i$. As $v_i$ was chosen without loss of generality the result follows by induction.
\end{proof}

\begin{lemma}\label{Lemma: Minus E(U) 2-dipath}
Let $G=(V,E)$ be an oriented graph and let $U \subseteq V$. Then, $$\chis(G) \leq |U| + \chis(G - E(G[U])).$$
\end{lemma}

\begin{proof}
   Let $G=(V,E)$ be an oriented graph, let $U = \{u_1,\dots, u_{|U|}\} \subseteq V$, and let $G' = G - E(G[U])$. 
   Suppose $\phi:V \rightarrow \{1,\dots, \chis(G')\}$ is a $2$-dipath colouring of $G'$. 
   We claim that $\psi: V \rightarrow \{1,\dots, |U|+\chis(G')\}$ defined by
   \begin{itemize}
       \item for $u_i\in U$, $\psi(u_i) = \chis(G')+i$, and
       \item for $w \in V\setminus U$, $\psi(w) = \phi(w)$,
   \end{itemize}
   is a $2$-dipath colouring of $G$. 

   For the sake of contradiction suppose $\psi$ is not a $2$-dipath colouring of $G$. 
   Then, there exists vertices $v,w \in V$ such that $1\leq \dist(v,w)\leq 2$ in $G$ and $\psi(v)=\psi(w)$.
   Suppose $v,w$ is such a pair of vertices.
   
   If $v,w\in V\setminus U$, then $\phi(v) = \psi(v) = \psi(w) = \phi(w)$.
   As $\phi$ is a $2$-dipath colouring of $G'$, this implies $\dist(v,w) > 2$ in $G'$.
   But this is a contradiction as $v,w\in V\setminus U$ implies that $v$ and $w$ have the same neighbours in $G'$ as in $G$.
   Hence, $\phi(v) = \psi(v) \neq \psi(w) = \phi(w)$ a contradiction.

   Suppose then that $v \in U$. 
   Then $v$ is the only vertex receiving colour $\psi(v)$ implying that if $\psi(v) = \psi(w)$, then $v=w$. 
   But this is a contradiction as $\dist(v,w)\geq 1$ in $G'$. This completes the proof.
\end{proof}

The following lemma was proven in \cite{bradshaw2023injective}.
For completeness a proof is included.

\begin{lemma}[Lemma~3.5 \cite{bradshaw2023injective}]\label{Lemma: Large Min Degree}
    Let $k \geq 1$ be an integer.
    If $G = (V,E)$ is a graph of Euler genus at most $g \geq 2$ and minimum degree at least $k + 6$, then $G$ has fewer than $\frac{6g}{k}$ vertices.
\end{lemma}
\begin{proof}
    Let $G = (V,E)$ be a graph of Euler genus at most $g \geq 2$ and $\delta(G) \geq k + 6$.
    By Euler's formula and the Handshaking lemma, $|V| - \frac{1}{3}|E| > -g$. Rearranging this,
    \[\sum_{v \in V} (\deg(v) - 6) < 6g.\]
    If each vertex has degree at least $k + 6$, then the number of terms in this sum is less than $\frac{6g}{k}$, completing the proof.
\end{proof}

The next step is to prove a smallest counterexample has large $2$-dipath chromatic number

\begin{lemma}\label{Lemma: H'-colouring, 2-dipath<= k, t-degen}
If $G = (V,E)$ is a smallest counterexample, then $\chis(G) > 138g -162$.
\end{lemma}

\begin{proof}
    Let $G = (V,E)$ be a smallest counterexample such that $\chis(G) \leq 138g-162$
    and let $v_1,\dots, v_n$ be a vertex order of $G$ such that for all $1\leq i \leq n$, $v_i$ is a vertex with minimum degree in $G[\{v_1,\dots, v_i\}]$.
    As  $\chis(G) \leq 138g-162$, there exists a $2$-dipath $(138g-162)$-colouring of $G$.
    Let $\psi: V \rightarrow \{1,\dots, 138g-162\}$ be a $2$-dipath colouring of $G$.

    By Lemma~\ref{Lemma: Small Graph are easy to "(k,d,N)-colour"} there exists an injective oriented homomorphism $\phi$ from $G[\{v_1,\dots, v_{6g}\}]$ to a graph $H$ which contains $H_g$ as a spanning subgraph, such that the image of $\phi$ is a subset of $P_0$.
    For all $6g \leq i \leq n$ we will construct an oriented homomorphisms $\phi_i$ from $G[\{v_1,\dots, v_i\}]$ to $H$ such that
    \begin{enumerate}
        \item if $j \leq 6g$, $\phi_i(v_j) \in P_0$ and for all $q \neq j$, $\phi_i(v_j) \neq \phi_i(v_q)$, and
        \item if $j > 6g$, then $\phi_i(v_j) \in P_{\psi(v_j)}$.
    \end{enumerate}
    This is well defined as for all $1\leq i \leq 138g-162$ there exists an independent set $P_i$ in $H$.
    Observe that as $\phi$ is injective and the image of $\phi$ is a subset of $P_0$, $\phi$ is a valid choice of $\phi_{6g}$. 
    Let $\phi_{6g} = \phi$, suppose $i > 6g$, and suppose $\phi_{i-1}$ has already been defined to satisfy properties (1) and (2).

    We now define $\phi_i$.
    For all $j < i$, let $\phi_i(v_j)=\phi_{i-1}(v_j)$.
    Observe that by Lemma~\ref{Lemma: Large Min Degree} $|N(v_i) \cap \{v_1,\dots, v_{i-1}\}| \leq 6$ as we have assumed $v_i$ is a vertex of minimum degree in $G[\{v_1,\dots, v_i\}]$.
    Let $A = \{a_1,\dots, a_{|A|}\}$ be an ordering of $\{\phi_i(u): u \in N(v_i)\cap \{v_1,\dots, v_{i-1}\}\}$.
    Then $|A|\leq 6$.

    Observe that if $v_{j},v_i,v_{q}$ are a $2$-dipath in $G$, $\psi(v_j)\neq \psi(v_q)$.
    Hence, assumptions (1) and (2) about $\phi_{i-1}$ implies that if $j,q < i$ and $v_{j},v_i,v_{q}$ for a $2$-dipath in $G$, then $\phi_{i-1}(v_j)\neq \phi_{i-1}(v_q)$.
    Thus, if $j,q < i$ and $v_{j},v_i,v_{q}$ for a $2$-dipath in $G$, then $\phi_{i}(v_j)\neq \phi_{i}(v_q)$.
    That is, vertices with different orientations to $v_i$ will receive different colours in $\phi_i$.

    As vertices with different orientations to $v_i$ receive different colours, the vector $\vec{v} \in \{-1,1\}^{|A|}$ with $x^{th}$ entry $1$ if $(v_i,u)\in E$, where $\phi(u) = a_x$, and $x^{th}$ entry $-1$ otherwise, is well defined. 
    Notice that if there exists an $s \in P_{\psi(v_i)}$ such that $\vec{v} = F(A,s,H)$, then letting $\phi_i(v_i)=s$ results in an oriented homomorphism $\phi_i$ from $G[\{v_1,\dots, v_i\}]$ to $H$ satisfying assumption (1) and (2).

    As assumption (1) and (2) are true for $\phi_{i-1}$, and $\phi_i(v_j) = \phi_{i-1}(v_j)$ 
    for each vertex $u \in N(v) \cap \{v_1,\dots, v_{i-1}\}$, 
    $\psi(u) \neq \psi(v_i)$ implies that $\phi(u) \notin P_{\psi(v_i)}$ for all $u \in N(v) \cap \{v_1,\dots, v_{i-1}\}$.
    Hence, $A \cap P_{\psi(v_i)} = \emptyset$ and $|A|\leq 6$.
    Thus, Lemma~\ref{Lemma: full H} implies that there exists a vertex $s \in P_{\psi(v_i)}$ such that $\vec{v} = F(A,s,H)$.
    Let $\phi_i(v_i)  = s$ for such a vertex $s$.
    Then, $\phi_i$ is an oriented homomorphism from $G[\{v_1,\dots, v_i\}]$ to $H$ satisfying assumption (1) and (2).
    
    Therefore, by induction for all $6g \leq i \leq n$ there exists oriented homomorphisms $\phi_i$ from $G[\{v_1,\dots, v_i\}]$ to $H$.
    But $G = G[\{v_1,\dots, v_n\}]$, so we have proven there exists an oriented homomorphism from $G$ to an oriented $H$ where $H_g$ is a spanning subgraph of $H$.
    This contradicts our assumption that $G$ is a counterexample, 
    so we conclude that in any smallest counterexample $G$, $\chis(G)> 138g-162$.
\end{proof}

\begin{lemma}\label{Lemma: smallest counterexample large max degree}
If $G=(V,E)$ is a smallest counterexample,
then $\Delta(G) > 12g-12$.
\end{lemma}

\begin{proof}
Suppose $G$ is a smallest counterexample with $\Delta(G) \leq 12g-12$. Let $v_1,\dots, v_n$ be a vertex ordering of $G$ such that for all $i$, $v_i$ is a vertex of minimum degree in $G[\{v_1,\dots, v_i\}]$. Then, Lemma~\ref{Lemma: Large Min Degree} implies that for all $i \geq 6g$, $v_i$ has degree at most $6$ in $G[\{v_1,\dots, v_i\}]$. Hence, $G' = G - E[G[\{v_1,\dots, v_{6g-1}\}]]$ is a $6$-degenerate graph.

As $\Delta(G') \leq \Delta(G) \leq 12g-12$, Lemma~\ref{Lemma: 2-Diapth Max Degree & Degen Upper Bound} implies that
$$
\chis(G')\leq 11(12g-12) - 29 = 132g - 161.
$$
Then, Lemma~\ref{Lemma: Minus E(U) 2-dipath} implies $\chis(G) \leq \chis(G') + 6g - 1 \leq 138g -162$. 
But this contradicts $G$ being a smallest counterexample as Lemma~\ref{Lemma: H'-colouring, 2-dipath<= k, t-degen} implies every smallest counterexample has $2$-dipath chromatic number strictly larger than $138g -162$.
We concludes that if $G$ is a smallest counterexample, then $\Delta(G) > 12g-12$.
\end{proof}

Finally, we are prepared to prove Theorem~\ref{Thm: Colouring Upper Bound}.

\begin{proof}[Proof of Theorem~\ref{Thm: Colouring Upper Bound}]
Let $G=(V,E)$ be a smallest counterexample.
We will prove that $G$ does not exist via a discharging argument.

We assign initial charge to vertices, $\omega: V \rightarrow \mathbb{N}$ such that $\omega(v) = \deg(v)-6$ for all $v\in V$.
As $G$ has Euler genus at most $g$, Euler's formula and the Handshacking lemma imply that
$$
\sum_{v\in V} \omega(v) = \sum_{v\in V} (\deg(v) -6) \leq 6g-12.
$$

We now define a discharging procedure and let $\omega^*: V \rightarrow \mathbb{N}$ be the resulting charge of each vertex in $G$. 
The procedure has two rules.
If $w$ is a vertex adjacent to a vertex $v$ such that $\deg(v)=4$, then $w$ gives charge $1/2$ to $v$.
If $w$ is a vertex adjacent to a vertex $v$ such that $\deg(v)=5$, then $w$ gives charge $1/5$ to $v$.
We claim that for every vertex $v$, $\omega^*(v) \geq 0$.
Recall that by Lemma~\ref{Lemma: <=3-vertex} every vertex in $G$ has degree greater than or equal to $4$.

If $v$ is a vertex with degree $4$, then every neighbour of $v$ has degree at least $12$ by Lemma~\ref{Lemma: ajacent low degree vertices}. 
Hence, if $v$ is a vertex with degree $4$, then $v$ receives charge $4 \times \frac{1}{2}$ 
and gives $0$ charge. 
Then, if  $v$ is a vertex with degree $4$, $\omega^*(v) = 4 - 6 + 2 = 0$.

If $v$ is a vertex with degree $5$, then every neighbour of $v$ has degree at least $12$ by Lemma~\ref{Lemma: ajacent low degree vertices}. 
Hence, if $v$ is a vertex with degree $5$, then $v$ receives charge $5 \times \frac{1}{5}$
and gives $0$ charge. 
Then, if  $v$ is a vertex with degree $5$, $\omega^*(v) = 5 - 6 + 1 = 0$.

If $v$ is a vertex with degree $6 \leq k \leq 11$, then $v$ is not adjacent to any degree $4$ or degree $5$ vertices by Lemma~\ref{Lemma: ajacent low degree vertices}. Hence, if $v$ has degree $6 \leq k \leq 11$, then $v$ gives $0$ charge and receives $0$ charge. 
Then, if  $v$ is a vertex with degree $6 \leq k \leq 11$, $\omega^*(v) = \omega(v) \geq 6-6 = 0$.

If $v$ is a vertex with degree $k\geq 12$, then $v$ is adjacent to at most $k$ vertices of degree $4$ or degree $5$.
Hence, $v$ receives charge $0$ and gives at most $k \times \frac{1}{2}$ charge.
Then, if $v$ is a vertex with degree $k\geq 12$, $\omega^*(v)\geq k - 6 - \frac{k}{2} \geq \frac{k}{2} - 6 \geq 0$.

This implies that for every vertex $v\in V$, $\omega^*(v) \geq 0$ as desired.
Noticing that the total charge of the graph has not changed after our discharging procedure,
\begin{align*}
    \sum_{v\in V} \omega^*(v) & = \sum_{v\in V} \omega (v) \leq 6g-12.
\end{align*}
As each term on the
left hand side is non-negative and for all $k \geq 12$ if $\deg(v)=k$, then $\omega^*(v) \geq \frac{k}{2} - 6$, we conclude that 
\begin{align*}
    \frac{\Delta(G)}{2} - 6 & \leq 6g -12.
\end{align*}

Therefore, $\Delta(G) \leq 12g-12$. But this contradicts $G$ being a smallest counterexample given Lemma~\ref{Lemma: smallest counterexample large max degree} states that if $G$ is a smallest counterexample, then $\Delta(G) > 12g-12$. 
We conclude that no smallest counterexample exists, thereby proving Theorem~\ref{Thm: Colouring Upper Bound}.
\end{proof}

\section{Oriented Colouring Graphs on Surfaces}

In this section we will prove that every graph with Euler genus at most $g$ can be oriented coloured with $O(g\log(g))$ colours.
The formal statement we will prove is Theorem~\ref{Thm: Upper Bound}.
Given we have already shown Theorem~\ref{Thm: Colouring Upper Bound} in Section~4, all that remains to be proven for Theorem~\ref{Thm: Upper Bound} is the existence of the graph $H$, and a bound on the order of the graph $H$ from Theorem~\ref{Thm: Colouring Upper Bound}.
Recall that $H$ is constructed from a $(144g-162,10,N)$-full graph $K^{(g)}$, 
and recall that $V(H) = V(K^{(g)})$.

Like Lemma~\ref{Lemma: Large Min Degree} the following lemma was prove in \cite{bradshaw2023injective}. 
For completeness the proof from \cite{bradshaw2023injective} is included.
The proof is an application of the probabilistic method. 

\begin{lemma}[Lemma~5.1 \cite{bradshaw2023injective}]\label{Lemma: Existence of Full graph}
    For each $d\geq 2$ and $k\geq 5$, there exists a $(k,d,\lceil 8^d \log k \rceil)$-full graph.
\end{lemma}

\begin{proof}
We let $N =  \lceil 8^d \log k \rceil $.
We let $H$ be a random orientation of the complete $k$-partite graph $K_{N,\dots,N}$ with partite sets $P_1,\dots, P_k$. 
Notice that if for all $i \in \{1,\dots, k\}$, 
all ordered subsets $U = \{u_1, \dots, u_d\}\subseteq \bigcup_{j \neq i} P_j$ of size $d$, 
and all vectors $\vec{v} \in \{-1,1\}^d$, 
there exists a $x \in P_i$ such that $F(U,x,G) = \vec{v}$,
then $H$ is $(k,d,N)$-full.

We consider a fixed value $i \in \{1,\dots, k\}$ and a
fixed ordered subset
$U = \{u_1, \dots, u_d\} \subseteq \bigcup_{j \neq i} P_j$, as well as a fixed vector $\vec{v} \in \{-1,1\}^d$.
The probability that a given vertex $x \in P_i$ satisfies $F(U,x,G) = \vec{v}$ is $2^{-d}$, so the probability that no vertex $v \in P_i$ satisfies $F(U,x,G) = \vec{v}$ is at most $(1-2^{-d})^N < \exp(-2^{-d} N)$. 
Therefore, taking a union bound over all possible values $i \in [k]$, all ordered subsets $U \subseteq \bigcup_{j \neq i} P_j$ of size $d$, and all vectors $\vec{v} \in \{-1,1\}^d$, the probability $p$ that $H$ is not $(k,d,N)$-full satisfies
\[ p \leq k \cdot (kN)^d 2^d \exp(-2^{-d}N).\]
The rest of the proof aims to show that $p < 1$. We observe that 
\begin{eqnarray*}
\log p &<& (d+1)(\log k + \log N + \log 2) - \frac{N}{2^d} \\
&= & (d+1) (\log k + \log
\lceil 8^d \log k \rceil + \log 2) - \frac{\lceil 8^d \log k \rceil }{2^d} \\ 
&< & (d+1)(2 \log k + \log 8^d + \log 2) - 4^d \log k \\
&=& (d+1) \left ( (2 - \frac{4^d}{d+1} ) \log k + (3d + 1) \log 2 \right )
\end{eqnarray*}
The $\log k$ term in the last expression has a negative coefficient for all $d \geq 2$, and therefore this expression is decreasing with respect to $k$.
 Hence, 
\[\log p < (d+1) \left ( (2 - \frac{4^d}{d+1} ) \log 5 + (3d + 1) \log 2 \right ),\]
which is negative for all $d \geq 2$.
Therefore, $p < 1$, and thus with positive probability, the oriented graph $H$ which we have constructed is $(k,d, \lceil 8^d \log k \rceil )$-full.
\end{proof}

\begin{theorem}\label{Thm: Upper Bound}
    For all integers $g\geq 2$,
    $$
    \chio(g) \leq 2^{40} g \log(g).
    $$
\end{theorem}

\begin{proof}
    Let $g\geq 2$ be an integer and let $G$ be an oriented graph of Euler genus at most $g$.
    Letting $k = 144g - 162$ and $d=10$ Lemma~\ref{Lemma: Existence of Full graph} implies there exists a 
    $$(144g - 162,10,\lceil 8^{10} \log(144g - 162) \rceil)\text{-full}$$
    graph. 
    Letting $K^{(g)}$ be a $(144g -162,10,N)$-full graph, where $N$ is chosen to be as small as possible, we see that $N \leq \lceil 8^{10} \log(144g - 162) \rceil$.

    Given $K^{(g)}$ we can define the graph $H_g$ as in Section~4.
    Note that $H_g$ has the same order as $K^{(g)}$.
    Then,
    $$
    |V(H_g)|\leq (144g - 162)\lceil 8^{10} \log(144g - 162) \rceil \leq 2^{40} g \log(g)
    $$
    for all $g \geq 2$.
    By Theorem~\ref{Thm: Colouring Upper Bound} $G$ has an oriented homomorphism to a oriented graph $H$ which contains $H_g$ as a spanning subgraph. Thus, 
    $$
    \chio(G) \leq |V(H)| = |V(H_g)| \leq 2^{40} g \log(g).
    $$
    As $G$ was chosen without loss of generality, this completes the proof.
\end{proof}

\section{Future Work}

Despite giving upper and lower bounds for $\chio(g)$ that are within a polylogarithmic factor of each other, the exact value of $\chio(g)$ for small $g$ remains open.
Even the maximum oriented chromatic number of planar graphs, $\chio(0)$, is poorly understood.
Any progress on the value of $\chio(g)$ for small $g$ is of interest.

One consequence of our proof that $\chio(g) \geq \Omega(\frac{g}{\log(g)})$ is that there exists an oriented clique with Euler genus at most $g$ whose order is asymptotically close to our upper bound for $\chio(g)$.
Is this a coincidence or part of a larger pattern for minor closed families?
We conjecture the latter.

\begin{conjecture}
    For a minor closed family of graphs $\mathcal{G}$ let $\omega_o(\mathcal{G})$ be order of a largest oriented clique in $\mathcal{G}$.
    There exists a function $f:\mathbb{N} \rightarrow \mathbb{N}$ such that $f(n) = n^{1+o(1)}$ and for all minor closed families of graphs $\mathcal{G}$, 
    $$
        \max_{G \in \mathcal{G}} \chio(G) \leq f(\omega_o(\mathcal{G})).
    $$
\end{conjecture}

\bibliographystyle{plain}
\bibliography{Bib}

\end{document}